\documentclass[12pt,reqno,a4paper]{amsart}
\usepackage{graphicx}
\usepackage{amssymb,amscd,textcomp}
\usepackage{hyperref}

\begin{document}

\let\kappa=\varkappa
\let\epsilon=\varepsilon
\let\phi=\varphi
\let\p\partial
\let\lle=\preccurlyeq
\let\ulle=\curlyeqprec

\def\Z{\mathbb Z}
\def\R{\mathbb R}
\def\C{\mathbb C}
\def\Q{\mathbb Q}
\def\P{\mathbb P}
\def\HH{\mathsf{H}}
\def\XX{\mathcal X}

\def\conj{\overline}
\def\Beta{\mathrm{B}}
\def\const{\mathrm{const}}
\def\ov{\overline}
\def\wt{\widetilde}
\def\wh{\widehat}

\renewcommand{\Im}{\mathop{\mathrm{Im}}\nolimits}
\renewcommand{\Re}{\mathop{\mathrm{Re}}\nolimits}
\newcommand{\codim}{\mathop{\mathrm{codim}}\nolimits}
\newcommand{\Aut}{\mathop{\mathrm{Aut}}\nolimits}
\newcommand{\lk}{\mathop{\mathrm{lk}}\nolimits}
\newcommand{\sign}{\mathop{\mathrm{sign}}\nolimits}
\newcommand{\rk}{\mathop{\mathrm{rk}}\nolimits}

\def\id{\mathrm{id}}
\def\Leg{\mathrm{Leg}}
\def\Jet{{\mathcal J}}
\def\sS{{\mathcal S}}
\def\lcan{\lambda_{\mathrm{can}}}
\def\ocan{\omega_{\mathrm{can}}}
\def\bgamma{\boldsymbol{\gamma}}

\renewcommand{\mod}{\mathrel{\mathrm{mod}}}

\newtheorem{mainthm}{Theorem}
\renewcommand{\themainthm}{{\Alph{mainthm}}}
\newtheorem{thm}{Theorem}[section]
\newtheorem{lem}[thm]{Lemma}
\newtheorem{prop}[thm]{Proposition}
\newtheorem{cor}[thm]{Corollary}

\theoremstyle{definition}
\newtheorem{exm}[thm]{Example}
\newtheorem{rem}[thm]{Remark}
\newtheorem{df}[thm]{Definition}

\numberwithin{equation}{section}

\title{Redshift and contact forms}
\author[Chernov \& Nemirovski]{Vladimir Chernov and Stefan Nemirovski}
\thanks{This work was partially supported by a grant from the Simons Foundation (\#\,513272 to Vladimir Chernov).
The second author was partially supported by SFB/TRR~191 of the DFG and by RFBR grant \textnumero 17-01-00592-a}
\address{Department of Mathematics, 6188 Kemeny Hall,
Dartmouth College, Hanover, NH 03755-3551, USA}
\email{Vladimir.Chernov@dartmouth.edu}
\address{%
Steklov Mathematical Institute, Gubkina 8, 119991 Moscow, Russia;\hfill\break
\strut\hspace{8 true pt} Mathematisches Institut, Ruhr-Universit\"at Bochum, 44780 Bochum, Germany}
\email{stefan@mi.ras.ru}

\begin{abstract}
It is shown that the redshift between two Cauchy surfaces in a globally hyperbolic spacetime
equals the ratio of the associated contact forms on the space of light rays of that
spacetime.
\end{abstract}

\maketitle

\section{Introduction}
Let $X$ be a spacetime, that is, a connected time-oriented Lorentz manifold~\cite[\S 3.1]{BEE}.
The Lorentz scalar product on $X$ will be denoted by $\langle\text{ },\!\text{ }\rangle$
and assumed to have signature $(+,-,\dots,-)$ with $n\ge 2$ 
negative spatial dimensions.

Suppose that $n_E$ (`emitter') and $n_R$ (`receiver') are two infinitesimal observers,
i.e.\ future-pointing unit Lorentz length vectors, at events $E, R\in X$ connected by a null 
geodesic $\gamma$. Then the {\it photon redshift\/} $z=z(n_E,n_R,\gamma)$ from $n_E$ to $n_R$ 
along $\gamma$ is defined by the formula
$$
1+z = \frac{\langle n_E,\dot\gamma(E)\rangle}{\langle n_R,\dot\gamma(R)\rangle}
$$
for any affine parametrisation of $\gamma$.
In other words, $1+z$ is the ratio of the frequencies of any lightlike particle travelling 
along $\gamma$ measured by $n_E$ and $n_R$, see e.g.~\cite[Appendix 9A]{HEL} or \cite[p.~354]{ONeill}.
If $z>0$, such particles appear `redder' (having lower frequency) to $n_R$ than to $n_E$, whence the terminology.

Assume now that $X$ is globally hyperbolic~\cite{BeSa1,BeSa2} and consider its space of light rays $\mathfrak N_X$.
By definition, a point $\bgamma\in\mathfrak N_X$ is an equivalence class of inextendible
future-directed null geodesics up to an orientation preserving affine reparametrisation.

A seminal observation of Penrose and Low~\cite{PR2, Lo1, Lo2} is that the space $\mathfrak N_X$
has a canonical structure of a contact manifold (see also \cite{NT, KT, BIL}). 
A contact form $\alpha_M$ on $\mathfrak N_X$ 
defining that contact structure can be associated to any smooth spacelike Cauchy 
surface~$M\subset X$. Namely, consider the map 
$$
\iota_M:\mathfrak N_X \mathrel{\lhook\joinrel\longrightarrow} T^*M
$$
taking $\bgamma\in\mathfrak N_X$ represented by a null geodesic~$\gamma\subset X$
to the $1$-form on $M$ at the point $x=\gamma\cap M$ 
collinear to ${\langle\dot\gamma(x),\cdot\,\rangle|}_M$
and having unit length with respect to the induced Riemann metric on~$M$ 
(see formula~\eqref{normalise} below). 
This map identifies $\mathfrak N_X$ with the unit cosphere bundle ${\mathbb S}^*M$
of the Riemannian manifold $\left(M,- {\langle\text{ },\!\text{ }\rangle|}_M\right)$. 
Then 
$$
\alpha_M := \iota_M^*\,\lambda_{\mathrm{can}},
$$
where $\lambda_{\mathrm{can}}=\sum p_kdq^k$ is the canonical Liouville $1$-form on $T^*M$.

Contact forms defining the same contact structure are pointwise proportional. 
The purpose of the present note is to point out that the ratio of the contact forms 
on $\mathfrak N_X$ associated to different Cauchy surfaces in $X$ is given by 
the redshifts between infinitesimal observers having those Cauchy surfaces as their rest spaces.

\begin{df}
Let $M$ and $M'$ be spacelike Cauchy surfaces in~$X$. 
The redshift from $M$ to $M'$ along $\bgamma\in{\mathfrak N}_X$ is defined by
$$
z(M,M',\bgamma):= z\bigl(n_M(x), n_{M'}(x'), \gamma\bigr),
$$
where $\gamma$ is any inextendible null geodesic representing $\bgamma$, 
$x=\gamma\cap M$, $x'=\gamma\cap M'$, and $n_M$ and $n_{M'}$ are the future pointing normal unit 
vector fields on $M$ and~$M'$.
\end{df}

\begin{thm}
\label{main}
Let $M$ and $M'$ be spacelike Cauchy surfaces in~$X$.
For every $\bgamma\in \mathfrak N_X$,  we have 
$$
\frac{\alpha_{M'}}{\alpha_{M}}(\bgamma) =
1 + z(M, M', \bgamma).
$$
\end{thm}

\begin{rem}
The theorem remains true for {\it partial\/} Cauchy surfaces, i.e.\ locally
closed acausal spacelike hypersurfaces $M,M'\subset X$, and for 
$\bgamma\in\mathfrak N_X$ corresponding to null geodesics
intersecting both $M$ and~$M'$.
\end{rem}

\begin{rem}
If $M$ and $M'$ are Cauchy surfaces through a point $x\in X$ such that $n_M(x)=n_{M'}(x)$,
then the theorem shows that the contact forms $\alpha_M$ and $\alpha_{M'}$ coincide on
the tangent spaces to $\mathfrak N_X$ at all points corresponding to null geodesics
passing through~$x$. In other words, an infinitesimal observer at an event $x$ defines
a contact form on $T{\mathfrak N}_X$ restricted to the sky~${\mathfrak S}_x\subset{\mathfrak N}_X$.
\end{rem}

The contact geometry of $\mathfrak N_X$ was previously used to recover the causal 
or, equivalently~\cite{Ma}, conformal structure of $X$, see~\cite{Lo2, NT, ChRu, ChNe1, ChNe2, ChNe3}.
Theorem~\ref{main} should make it possible to apply techniques from contact geometry 
to study the metric structure of a globally hyperbolic spacetime. 
A token application to the comparison of Liouville and Riemannian volumes 
on different Cauchy surfaces is given in \S\ref{SectVolume} below.

\section{Proof of Theorem~\ref{main}}
The key fact is the following basic property of vector fields tangent to variations
of pseudo-Riemannian geodesics by curves of the same speed. For Jacobi fields tangent
to families of null geodesics in Lorentz manifolds, this computation appears 
in~\cite[p.~176]{PR2}, \cite[pp.~252--253]{NT}, and~\cite[pp.~10--11]{BIL}.

\begin{lem} 
\label{constant}
Let $\gamma_s:(a,b)\to X$, $0\le s <\epsilon$, be a one-parameter family of 
curves in a pseudo-Riemannian manifold $(X,\langle\text{ },\!\text{ }\rangle)$
such that $\gamma_0$ is a geodesic and $\langle \dot\gamma_s, \dot\gamma_s\rangle$ is independent of~$s$. 
If
$$
J(t) := \left.\frac{d}{ds}\right|_{s=0} \gamma_s(t)
$$
is the vector field along $\gamma_0$ tangent to this family, then 
$$
\langle \dot\gamma_0(t), J(t) \rangle = \mathrm{const.}
$$
\end{lem}

\begin{proof} 
Let us show that the $t$-derivative of this scalar product is zero. Indeed,
\begin{equation}
\label{derive}
\frac{d}{dt} \langle\dot\gamma_0(t), J(t)\rangle = \langle\nabla_t \dot\gamma_0(t), J(t)\rangle + \langle\dot\gamma_0(t), \nabla_t J(t)\rangle,
\end{equation}
where $\nabla$ is the pull-back of the Levi-Civita connection of the pseudo-Riemannian metric on~$X$
to $(a,b)\times [0,\epsilon)$ by the map $(t,s)\mapsto \gamma_s(t)$.
The first term on the right hand side vanishes because the tangent vector
of a geodesic is parallel along the geodesic. Note further
that
$$
\nabla_t \frac{\partial}{\partial s} = \nabla_s \frac{\partial}{\partial t}
$$
since the Levi-Civita connection has no torsion and $\left[\frac{\partial}{\partial s},\frac{\partial}{\partial t}\right]=0$
(see also~\cite[Proposition 4.44(1)]{ONeill}). Hence, the right hand side of~\eqref{derive} is equal to
$$ 
{\langle \dot\gamma_0(t), \nabla_s \dot\gamma_s(t)\Big{|}}_{s=0}\rangle  
= \frac{1}{2}\left.\frac{d}{ds}\right|_{s=0}\langle\dot\gamma_s,\dot\gamma_s\rangle = 0
$$
because the speed of $\gamma_s$ is independent of $s$ by assumption.
\end{proof}

\begin{rem}
The relevance of torsion in this context is 
pointed out in the footnote on p.~184 of~\cite{PR2}.
\end{rem}

Let now $M$ be a smooth spacelike Cauchy surface in a spacetime~$X$ and 
$\gamma$ an inextendible future-directed null geodesic in $X$ intersecting $M$
at the (unique) point $x=\gamma\cap M$. Then
\begin{equation}
\label{normalise}
\iota_M(\bgamma) = \frac{{\langle \dot\gamma(x),\cdot\,\rangle|}_M}{\langle \dot\gamma(x), n_M(x)\rangle}\, ,
\end{equation}
where $n_M$ is the future-pointing unit normal vector field on $M$ 
and $\bgamma\in\mathfrak N_X$ is the equivalence class of~$\gamma$.
Indeed, since $\langle \dot\gamma(x),\cdot\,\rangle$ is a null covector,
the Riemannian length of its restriction to $T_xM$ is equal
to the Riemannian length of its restriction to the Lorentz normal direction,
which is precisely $\langle \dot\gamma(x), n_M(x)\rangle (>0)$.

Thus, if $\mathbf{v}\in T_{\bgamma}\mathfrak N_X$ and $v={(\iota_M)}_*\mathbf{v}$,
then 
\begin{equation}
\label{evaluate}
\alpha_M(\mathbf{v}) = \lambda_\mathrm{can}(v) =
\frac{{\langle \dot\gamma(x),{(\pi_M)}_* v\rangle}}{\langle \dot\gamma(x), n_M(x)\rangle}
\end{equation}
by the definition of the canonical $1$-form $\lambda_\mathrm{can}$ and formula~\eqref{normalise}, 
where $\pi_M:T^*M\to M$ denotes the bundle projection.

Suppose that $\gamma_s:(a,b)\to X$, $s\in [0,\epsilon)$, is a family of null geodesics 
intersecting $M$ such that the maximal extension of $\gamma_0$ is $\gamma$ and the 
corresponding curve in $\mathfrak N_X$ has tangent vector $\mathbf{v}$ at $\bgamma$ 
or, equivalently, ${\left.\frac{d}{ds}\right|}_{s=0}\iota_M(\bgamma_s) = v$. 
Let $x(s)=\gamma_s\cap M$ so that $x(0)=x$. Then 
$$
{(\pi_M)}_* v = {\left.\frac{d}{ds}\right|}_{s=0} x(s)
$$ 
because $x(s)=\pi_M\circ\iota_M(\bgamma_s)$ by the definition of $\iota_M$. Hence, 
$$
{(\pi_M)}_* v = J(x) + \tau'(0)\dot\gamma(x),
$$
where $J = \left.\frac{d}{ds}\right|_{s=0} \gamma_s$ is the Jacobi vector field
along $\gamma_0$ tangent to the family $\gamma_s$ and $\tau=\tau(s)$ is the
function defined by $\gamma_s(\tau(s))=x(s)$.
Since $\dot\gamma(x)$ is null, 
it follows that
\begin{equation}
\label{Jacobi}
{\langle \dot\gamma(x),{(\pi_M)}_* v\rangle} = {\langle \dot\gamma(x),J(x)\rangle}.
\end{equation}

If $M'$ is another Cauchy surface and $x'=\gamma\cap M'$, we may choose $(a,b)\subseteq\R$
so that $\gamma(a,b)\ni x, x'$ and a family $\gamma_s$ as above exists on~$(a,b)$.
By formulas~\eqref{evaluate} and~\eqref{Jacobi}, we obtain
$$
\alpha_M(\mathbf{v}) = \frac{\langle \dot\gamma(x),J(x)\rangle}{\langle \dot\gamma(x), n_M(x)\rangle}
\qquad\text{and}\qquad
\alpha_{M'}(\mathbf{v}) = \frac{\langle \dot\gamma(x'),J(x')\rangle}{\langle \dot\gamma(x'), n_{M'}(x')\rangle}\, .
$$
However, 
$$
\langle \dot\gamma(x),J(x)\rangle = \langle \dot\gamma(x'),J(x')\rangle
$$
by Lemma~\ref{constant} and therefore
$$
\frac{\alpha_{M'}(\mathbf{v})}{\alpha_M(\mathbf{v})} = \frac{\langle \dot\gamma(x), n_M(x)\rangle}{\langle \dot\gamma(x'), n_{M'}(x')\rangle}
= 1 + z\bigl(n_M(x), n_{M'}(x'), \gamma\bigr),
$$
which proves Theorem~\ref{main}.

\begin{rem}
The proof shows that the ratio $\frac{\alpha_{M'}(\mathbf{v})}{\alpha_M(\mathbf{v})}$, where $\mathbf{v}$ is a tangent vector to $\mathfrak N_X$ 
at a point $\bgamma\in \mathfrak N_X$, is a positive function depending only on~$\bgamma$. Thus, the contact forms on 
$\mathfrak N_X$ associated to different Cauchy surfaces in~$X$ define the same co-oriented contact structure indeed. 
This contact structure can also be described as the pull-back of the canonical contact structure on the spherical cotangent bundle $ST^*M$ 
of a Cauchy surface $M$ by the map $\rho_M = s_M\circ\iota_M$, where $s_M: T^*M-\{\text{zero section}\}\to ST^*M$ is the projection to the spherisation, 
see~\cite[pp.~252--253]{NT} and~\cite[\S 4]{ChNe1}.
\end{rem}

\section{Liouville measure and Riemannian volume}
\label{SectVolume}
Let $M$ and $M'$ be two spacelike Cauchy surfaces in a globally hyperbolic
spacetime $(X, \langle\text{ },\!\text{ }\rangle)$ and consider 
the contact forms $\alpha_M=\iota_M^*\lambda_{\mathrm{can}}$ and $\alpha_{M'}=\iota_{M'}^*\lambda_{\mathrm{can}}$
on ${\mathfrak N}_X$ associated to $M$ and $M'$.
Then 
$$
\alpha_{M} = \bigl(1+z(M,M',\bgamma)\bigr)^{-1}\alpha_{M'}
$$
by Theorem~\ref{main} and therefore
\begin{equation}
\label{VolumeNullGeo}
\alpha_M\wedge\left(d\alpha_M\right)^{n-1} = \bigl(1+z(M,M',\bgamma)\bigr)^{-n}\alpha_{M'}\wedge\left(d\alpha_{M'}\right)^{n-1}
\end{equation}
because $\alpha\wedge\alpha=0$ and $d(f\alpha)=fd\alpha + df\wedge\alpha$ for any function~$f$ and 1-form~$\alpha$.

Recall that the {\it Liouville measure\/} on the unit cosphere bundle of a Riemannian manifold is defined 
by the non-vanishing $(2n-1)$-form
$$
\Omega:=\lambda_{\mathrm{can}}\wedge\left( d\lambda_{\mathrm{can}}\right)^{n-1}.
$$
Thus, formula~\eqref{VolumeNullGeo} may be viewed as a general volume--redshift relation (cf.~\cite[\S14.12 and \S15.9]{HEL})
for the Liouville measures on the unit cosphere bundles of $M$ and $M'$ with respect to the Riemann metrics
$-{\langle\text{ },\!\text{ }\rangle|}_M$ and $-{\langle\text{ },\!\text{ }\rangle|}_{M'}$. 
Indeed, let
\begin{equation}
\label{LightRaysMap}
\iota_{M'M} = \iota_M\circ(\iota_{M'})^{-1} : {\mathbb S}^*M'\overset{\cong}{\longrightarrow} {\mathbb S}^*M
\end{equation}
be the map identifying the unit covectors corresponding to the same null geodesic
at its intersection points with $M$ and $M'$. Then \eqref{VolumeNullGeo} shows that
\begin{equation}
\label{LiouvilleRedshift}
(\iota_{M'M})^*\Omega_M = \bigl(1+z(M,M',\bgamma)\bigr)^{-n}\Omega_{M'}
\end{equation}
at $\iota_{M'}(\bgamma)\in {\mathbb S}^*M'$.

Let $\mathfrak{L}\subseteq\mathfrak{N}_X$ be a (Borel) subset of the space of light rays and denote by $\mathfrak{L}_x$
the set of null geodesics from $\mathfrak{L}$ passing through a point~$x\in X$. Integrating~\eqref{LiouvilleRedshift} 
over $\iota_{M'}(\mathfrak{L})$, we obtain that
\begin{equation}
\label{LiouvilleRedshiftInt}
\int\limits_{\iota_M(\mathfrak{L})} \Omega_M 
= \int\limits_{\iota_{M'}(\mathfrak{L})} \bigl(1+z(M,M',\bgamma)\bigr)^{-n} \Omega_{M'}.
\end{equation}
The Liouville measure is locally the product of the Riemann measure on the base manifold and the surface area measure
on the unit sphere in the standard Euclidean space~$\R^n$, see \cite[\S5.2]{B} or \cite[Theorem~VII.1.3]{Ch}. 
Therefore both integrals in~\eqref{LiouvilleRedshiftInt} can be converted to double integrals.
Applying this to the left hand side first, we see that
$$
\int\limits_{\iota_M(\mathfrak{L})} \Omega_M = \int\limits_{M} dV_M(x) \int\limits_{\iota_M(\mathfrak{L}_x)} d\omega_x
= \int\limits_{M} \omega_M(x,\mathfrak{L})\, dV_M(x),
$$
where $dV_M$ is the Riemann measure on $M$, $d\omega_x$ is the surface area measure on the fibre $\mathbb{S}_x^*M$, 
and 
$$
\omega_M(x,\mathfrak{L}) := \int\limits_{\iota_M(\mathfrak{L}_x)} d\omega_x
$$ 
is the area of the set $\iota_M(\mathfrak{L}_x)$ of unit covectors at $x\in M$ 
corresponding to null geodesics from~$\mathfrak{L}$, i.e.\
the solid angle spanned by the light rays from $\mathfrak{L}$ at $x\in M$. 
Now \eqref{LiouvilleRedshiftInt} takes the form 
\begin{equation}
\label{LiouvilleRedshiftFibreInt}
\int\limits_{M} \omega_M(x,\mathfrak{L})\, dV_M(x) = 
\int\limits_{M'} dV_{M'}(x') \int\limits_{\iota_{M'}(\mathfrak{L}_{x'})} \bigl(1+z(M,M',\bgamma)\bigr)^{-n} d\omega_{x'}.
\end{equation}

\begin{exm}
\label{ExConst}
Assume that the redshift $z(M,M',\bgamma)= z$ is the same for all $\bgamma\in\mathfrak L$. Then \eqref{LiouvilleRedshiftFibreInt}
simplifies to
$$
\int\limits_{M} \omega_M(x,\mathfrak{L})\, dV_M(x)  = \frac{1}{(1+z)^n}\int\limits_{M'} \omega_{M'}(x',\mathfrak{L})\, dV_{M'}(x'). 
$$
\end{exm}

\begin{exm}
\label{AllRays}
Let $\mathfrak{L}=\mathfrak{N}_X$ be the set of all light rays. Then
$$
\iota_M(\mathfrak{L}_x) = \mathbb{S}^*_xM
$$
for every Cauchy surface $M$ and every point~$x\in M$. Hence,
$$
\omega_M(x,\mathfrak{L}) = \mathfrak{c}_n,
$$
where $\mathfrak{c}_n$ is the area of the standard unit sphere in~$\R^n$.
Therefore \eqref{LiouvilleRedshiftFibreInt} implies
$$
\mathfrak{c}_n \mathrm{Vol}(M) = 
\int\limits_{M'} dV_{M'}(x') \int\limits_{\mathbb{S}^*_{x'}M'} \bigl(1+z(M,M',\bgamma)\bigr)^{-n} d\omega_{x'}.
$$
If the redshift is constant as in Example~\ref{ExConst}, it follows that
$$
\mathrm{Vol}(M) = \frac{1}{(1+z)^n} \mathrm{Vol}(M').
$$
More generally, if $\underline{z}\le z(M,M',\bgamma) \le \overline{z}$, then
$$
\frac{1}{(1+\overline{z})^n} \mathrm{Vol}(M') \le \mathrm{Vol}(M)\le \frac{1}{(1+\underline{z})^n} \mathrm{Vol}(M').
$$
\end{exm}

\begin{exm}
\label{VolumeEx}
Consider a subset $D\subseteq M$ and let
$$
\mathfrak{L}^D = \{\bgamma\in\mathfrak{N}_X\mid \gamma\cap D\ne\varnothing\}
$$
be the set of all light rays passing through $D$. Then
$$
\iota_M(\mathfrak{L}_x^D) = 
\begin{cases} 
\mathbb{S}^*_xM, & x\in D,\\
\varnothing, & x\in M\setminus D,
\end{cases}
$$
and therefore
$$
\omega_M(x,\mathfrak{L}^D) = 
\begin{cases} 
\mathfrak{c}_n, & x\in D,\\
0, & x\in M\setminus D.
\end{cases}
$$
Hence, \eqref{LiouvilleRedshiftFibreInt}
gives the following expressions for the volume of $D$ in~$M$:
\begin{align}
\label{VolumeRedshift}
\mathrm{Vol}_{M}(D) &
=\frac{1}{\mathfrak{c}_n}
\int\limits_{M'} dV_{M'}(x') \int\limits_{\iota_{M'}(\mathfrak{L}^D_{x'})} \bigl(1+z(M,M',\bgamma)\bigr)^{-n} d\omega_{x'}\\
& = \frac{1}{\mathfrak{c}_n} \int\limits_{\{\iota_{M'}(\bgamma) \mid\, \gamma\cap D\neq\varnothing\}} 
\bigl(1+z(M,M',\bgamma)\bigr)^{-n}\Omega_{M'}.
\end{align}
Thus, the volume of $D\subseteq M$ can be computed by integrating the redshift factor $\bigl(1+z(M,M',\bgamma)\bigr)^{-n}$ 
with respect to the Liouville measure on $\mathbb{S}^*M'$ over the subset of all unit covectors on $M'$ corresponding to light 
rays $\bgamma$ passing through~$D$. For constant redshift, \eqref{VolumeRedshift} reduces to
$$
\mathrm{Vol}_{M}(D) = \frac{1}{\mathfrak{c}_n (1+z)^n}\int\limits_{M'} \omega_{M'}(x',\mathfrak{L}^D)\, dV_{M'}(x').
$$
Note that if $M$ lies in the past of~$M'$, then $\omega_{M'}(x',\mathfrak{L}^D)$ may be interpreted 
as the solid angle at $x'\in M'$ subtended by $D\subseteq M$.
\end{exm}

\begin{figure}[htbp]
\begin{center}
\includegraphics[scale=0.6]{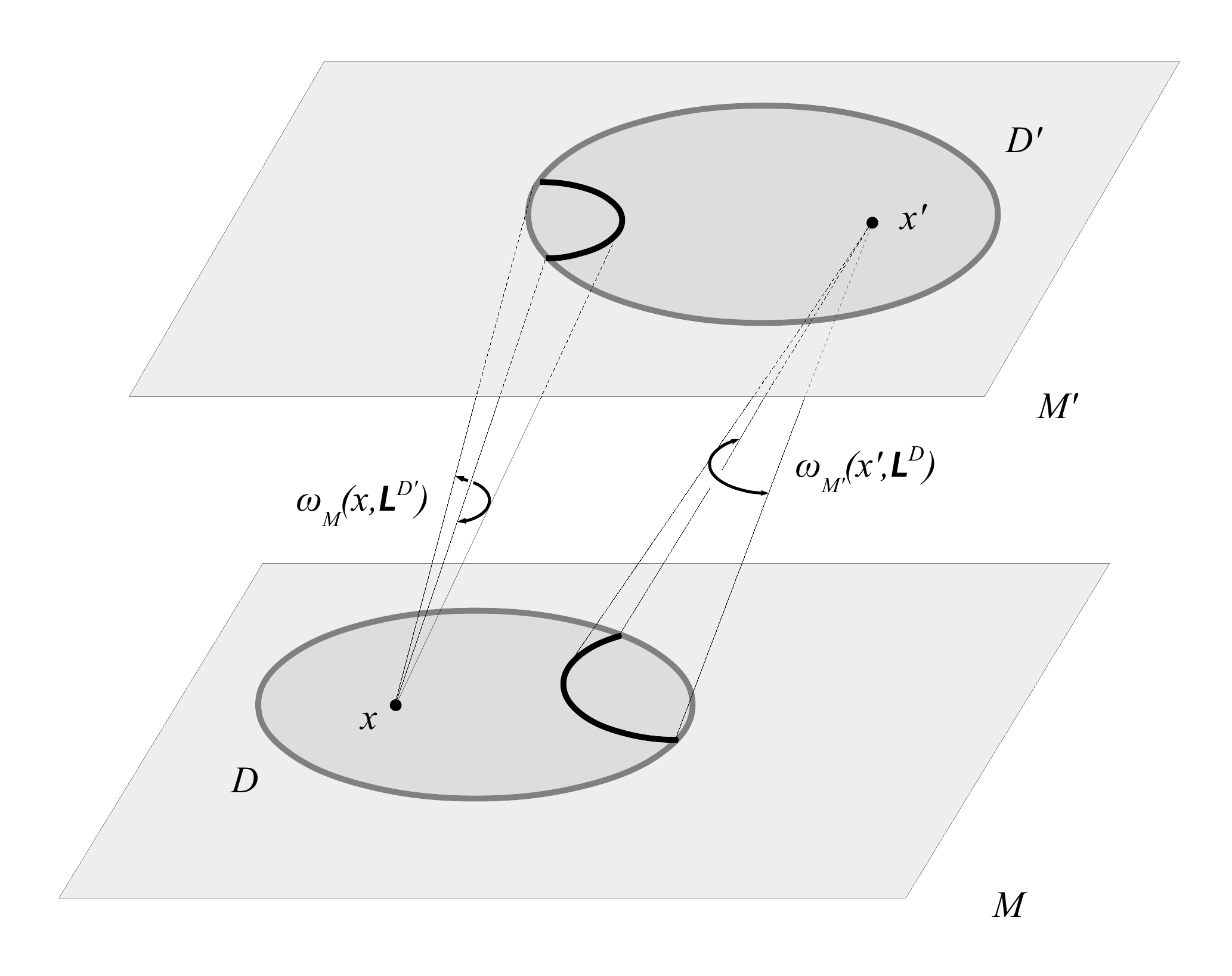}
\end{center}
\caption{Cauchy surfaces and light rays ($n=2$).} 
\label{angles}
\end{figure}

\begin{exm} Let now
$$
\mathfrak{L}^{DD'}:= \mathfrak{L}^D \cap \mathfrak{L}^{D'} =  \{\bgamma\in\mathfrak{N}_X\mid \gamma\cap D\ne\varnothing, \gamma\cap D'\ne\varnothing\}
$$
be the set of all light rays intersecting $D\subseteq M$ and $D'\subseteq M'$. 
(Example~\ref{VolumeEx} is a special case of this situation with $D'=M'$.)
Then
$$
\iota_M(\mathfrak{L}_x^{DD'}) = 
\begin{cases} 
\iota_M(\mathfrak{L}_x^{D'}), & x\in D,\\
\varnothing, & x\in M\setminus D,
\end{cases}
$$
and similarly
$$
\iota_{M'}(\mathfrak{L}_{x'}^{DD'}) = 
\begin{cases} 
\iota_{M'}(\mathfrak{L}_{x'}^{D}), & x'\in D',\\
\varnothing, & x'\in M'\setminus D'.
\end{cases}
$$
Hence, it follows from~\eqref{LiouvilleRedshiftFibreInt} that
$$
\int\limits_{D} \omega_M(x,\mathfrak{L}^{D'})\, dV_M(x) = 
\int\limits_{D'} dV_{M'}(x') \int\limits_{\iota_{M'}(\mathfrak{L}^{D}_{x'})} \bigl(1+z(M,M',\bgamma)\bigr)^{-n} d\omega_{x'}.
$$
In the case of constant redshift~$z$, we obtain
$$
\int\limits_{D} \omega_M(x,\mathfrak{L}^{D'})\, dV_M(x)
= \frac{1}{(1+z)^n}\int\limits_{D'} \omega_{M'}(x',\mathfrak{L}^{D})\, dV_{M'}(x').
$$
If $M$ is in the past of $M'$, then $\omega_{M'}(x',\mathfrak{L}^{D})$ is the solid
angle subtended by $D$ at $x'$ as in Example~\ref{VolumeEx} and $\omega_M(x,\mathfrak{L}^{D'})$
is the solid angle at $x\in M$ spanned by rays emitted from $x$ and received in~$D'$,
see Fig.~\ref{angles}.
\end{exm}

\end{document}